\documentclass[a4paper,reqno]{amsart}
\usepackage[english]{babel}
\usepackage{amsmath,amsthm,amssymb,amsfonts}

\usepackage{enumitem}
\usepackage{comment}
\usepackage{hyperref}
\usepackage{mathtools}
\usepackage[T1]{fontenc}
\usepackage[utf8]{inputenc}
\usepackage{dsfont}
\hypersetup{
 colorlinks,
 linkcolor={blue!90!black},
 citecolor={red!80!black},
 urlcolor={blue!50!black}
}
\usepackage{tikz}
\usepackage{tikz-cd}
\usepackage{graphicx}
\usepackage[all,cmtip]{xy}
\usepackage{mleftright}
\usepackage{booktabs}
\usepackage{cite}

\newtheorem{theorem}{\sc \textbf{Theorem}}[section]  
\newtheorem{proposition}[theorem]{\sc \textbf{Proposition}}   
\newtheorem{corollary}[theorem]{\sc \textbf{Corollary}}        
\newtheorem{lemma}[theorem]{\sc \textbf{Lemma}}

\theoremstyle{remark}
\newtheorem{definition}[theorem]{\sc \textbf{Definition}}

\def\cM{{\mathcal M}}

\def\cR{{\mathcal R}}

\numberwithin{equation}{section}

\usepackage{pdfpages}

\newcommand{\Lpmuw}{L_\mu^p(w)}

\makeatletter
\@namedef{subjclassname@1991}{Mathematics Subject Classification}
\usepackage{enumerate}
\graphicspath{{images/}}
\title[$A_p$ weights on nonhomogeneous trees]{$A_p$ weights on nonhomogeneous trees equipped with measures of exponential growth}
\author[A. Ottazzi]{Alessandro Ottazzi}
\address[A. Ottazzi]{School of Mathematics and Statistics \\  University of New South
Wales \\ 2052 Sydney  \\  Australia}
\email{a.ottazzi@unsw.edu.au}
\author[F. Santagati]{Federico Santagati}
\address[F. Santagati]{Dipartimento di Matematica e Applicazioni \\   Universit\`a degli Studi di Milano-Bicocca  \\ Via Cozzi 55 \\ 20125 Milano \\ Italy}
\email{federico.santagati@unimib.it}

\author[M. Vallarino]{Maria Vallarino}
\address[M. Vallarino]{Dipartimento di Scienze Matematiche ``Giuseppe Luigi Lagrange'' \\  Politecnico di Torino \\ Corso Duca degli Abruzzi 24 \\ 10129 Torino \\ Italy}
\email{maria.vallarino@polito.it}
\thanks{Santagati and Vallarino are members of the Project ``Harmonic analysis on continuous and discrete structures'' (bando Trapezio Compagnia di San Paolo CUP E13C21000270007) and of the Gruppo Nazionale per l'Analisi Matema\-tica, la Probabilit\`a e le loro Applicazioni (GNAMPA) of the Istituto Nazionale di Alta Matematica (INdAM)}
\begin{document}  \maketitle
\begin{abstract}
 This paper aims to study $A_p$ weights in the context of a class of metric measure spaces with exponential volume growth, namely infinite trees with root at infinity equipped with the geodesic distance and flow measures. Our main result is a Muckenhoupt Theorem, which is a characterization of the weights for which a suitable Hardy--Littlewood maximal operator is bounded on the corresponding weighted $L^p$ spaces.  We emphasise that this result does not require any geometric assumption on the tree or any condition on the flow measure. We also prove a reverse H\"older inequality in the case when the flow measure is locally doubling. We finally show that the logarithm of an $A_p$ weight is in BMO and discuss the connection between $A_p$ weights and quasisymmetric mappings.
 \end{abstract}

\section{Introduction}
The study of $A_p$ weights originated with Muckenhoupt's seminal work \cite{Mu1}, which initially characterized these weights in $\mathbb{R}^n$ in terms of the boundedness of the Hardy–Littlewood maximal operator from the weighted $L^p$ space to itself. Subsequently, Coifman and Fefferman introduced the fundamental concept of the reverse H\"older inequality \cite{CF} and Calderón extended these results to spaces of homogeneous type  \cite{Ca}. Moreover, $A_p$ weights were further characterized in $\mathbb{R}$ by linking them to the boundedness of the Hilbert transform in weighted $L^p$ spaces, as shown in \cite{HMW}. This line of inquiry led to a comprehensive theory regarding the boundedness of Calderón-Zygmund operators in weighted $L^p$ spaces; we refer to \cite{Muckenhoupt1974, GCRdF, Du, DO, Gra} and the references therein for more information on the topic.   \\ 
\indent This work focuses on trees, which are infinite connected graphs with no loops. It is important to note that a tree, when equipped with the usual geodesic distance and the counting measure, is a nondoubling metric measure space in most cases.  In this context, Muckenhoupt's characterization of $A_p$ weights fails. Indeed, on the one hand, the $A_p$ condition may be too weak, given that on some trees the Hardy--Littlewood maximal operator can be unbounded on every $L^p$ for $p<\infty$ \cite{LMSV}. On the other hand, it can be too strong since there are other examples of trees on which the $A_p$ condition is not necessary for the boundedness of the maximal operator \cite{OR-MathAnn}. For this reason, we introduce different measures on trees, the {\it flow measures}, where we will be able to prove a Muckenhoupt's characterization of $A_p$ weights. Before providing the details of our results,
some preliminaries are in order. \\ 
Let $T$ be a tree and $d$ denote the geodesic distance. We say that two vertices $x,y \in T$ are neighbours if and only if $d(x,y)=1$. In this case, we write $x \sim y$. Let $\partial T$ be the boundary of $T$ (defined, e.g., as in \cite[Section I.1]{FTN}). We fix a reference point $\omega_*$ in the boundary of $T$ and think of $T$ hanging down from $\omega_*$. This choice induces a partial order $\le$ on $T$, namely, 
$x \le y$ if and only if $y \in [x,\omega_*)$, where $[x,\omega_*)$ denotes the semi-infinite geodesic with endpoints $x$ and $\omega_*$.
Fix a reference vertex $o \in T$ and  let $\{x_j\}_{j \in \mathbb N}$ be an enumeration of $[o,\omega_*)$ such that $x_0=o$ and $d(x_j,x_k)=|j-k|$.
We define the  level function \begin{align}\label{def:lev}\ell(x)=\lim_{j \to \infty} j - d(x,x_j), \qquad \forall x \in T.
\end{align}For every $x \in T,$ let $s(x)$ denote the set of {\it successors} of $x$, i.e,  $$s(x):=\{y \sim x \ : \ell(y)=\ell(x)-1\}$$ and $p(x)$ the {\it predecessor} of $x$, namely, $$p(x):=\{y \sim x \ : \ell(y)=\ell(x)+1\}.$$  We call a positive function $\mu$ on $T$ a flow measure if it satisfies  the flow condition \begin{align}
    \mu(x)=\sum_{y \in s(x)} \mu(y).
\end{align} 
We point out that, for a general flow measure $\mu$, $(T, d, \mu)$ is not a doubling metric measure space and the measure of balls grows at least exponentially (see \cite[Theorem 2.5 and Proposition 2.8]{LSTV}). Since the doubling condition fails in this setting one has to identify a suitable definition of $A_p$ weights. 

Let a flow measure $\mu$ on $T$ be fixed. Inspired by a similar definition given in the setting of homogeneous trees in \cite{HS}, in \cite{LSTV} the authors introduced a family $\mathcal R$ of {\it admissible trapezoids} in $(T,d,\mu)$ (see Section \ref{Sec:2} for their precise definition) and used those sets to replace balls and develop a Calder\'on--Zygmund theory in this setting, at least when $\mu$ is locally doubling. For this reason, it is quite natural to give a definition of $A_p(\mu)$, $p\in [1,\infty)$, on $(T,d,\mu)$ which is similar to the classical one, where balls are replaced by admissible trapezoids. In Section \ref{Sec:2} we study the properties of such weights and prove a result in the spirit of Muckenhoupt Theorem, showing that the Hardy--Littlewood maximal operator $\mathcal M_{\mu}$ associated to the family of admissible trapezoids and the flow measure $\mu$ is bounded on the weighted $L^p_{\mu}(w)$ space if and only if $w\in A_p(\mu)$ when $p\in (1,\infty)$. Similarly, $\mathcal M_{\mu}$ is bounded from $L^1_\mu(w)$ to $L^{1,\infty}_\mu(w)$ if and only if $w\in A_1(\mu)$ (Theorem~\ref{Muckenhoupt-Main}). It is remarkable that for this result we do not require any condition on the geometry of the tree, which may have unbounded degree and even not be locally finite, or on the measure $\mu$ which does not need to be locally doubling. A key ingredient to study the boundedness properties of the operator $\mathcal M_{\mu}$ is the fact that every admissible trapezoid $R$ admits an envelope set $\widetilde R$ which plays a role in a covering lemma and has flow measure comparable with the flow measure of $R$. Unfortunately, $\widetilde R$ is not admissible any longer; this makes quite involved to show that also the weighted measure of $\widetilde R$ and $R$ are comparable for every $A_p(\mu)$ weight (see Theorem \ref{th1}).

\smallskip

In Section \ref{secRevHolder} we restrict our attention to the case when $\mu$ is locally doubling, which implies that the tree has uniformly bounded degree. Under these conditions, we prove a reverse H\"older inequality for $A_p(\mu)$ weights, and we provide a characterization of the class $A_\infty(\mu)$, which we define as the union of all $A_p(\mu)$ spaces as $p>1$. One difficulty that we have to face is that in our setting the measure $w \mu$ is not a flow and is nondoubling for a general weight $w\in A_p(\mu)$, so that a Calder\'on--Zygmund decomposition for such a measure is not available. We tackle this challenge in Proposition~\ref{newCZdec} where, despite the lack of a Calder\'on--Zygmund theory for $w\mu$, we are able to control the averages of a function on admissible trapezoids with respect to such measure. As in the classical setting, we also prove that the space $BMO(\mu)$, defined in this context in \cite{LSTV}, coincides with the space of multiples of logarithms of $A_{\infty}(\mu)$ weights.

\smallskip

In Section~\ref{QSmappings}, we discuss the connection between $A_p(\mu)$ weights and quasisymmetric mappings. A classical theorem of Reimann~\cite{Reimann} states that the Jacobian determinant of a quasiconformal mapping on $\mathbb R^n$ is an $A_p$ weight, which in turn implies that its logarithm is in ${\rm BMO}(\mathbb R^n)$. 
Since Reimann's result was published, several attempts to generalise his results to different metric measure spaces were made, see, e.g., \cite{Korte}. Recently, in \cite{NguyenWard}, the authors proved that the logarithm of the generalised Jacobian of a quasisymmetric mapping on spaces of homogeneous type is always in BMO. Following Reimann's strategy, they first prove that the generalised Jacobian of a quasisymmetric mapping is an $A_p$ weight. The proof relies heavily on the doubling property. 
It turns out that with the geodesic distance on the tree, the (suitably defined) Jacobian of a quasisymmetric mapping is not an $A_\infty(\mu)$ weight, implying that Reimann's result is false in this context. In fact, we will show (Theorem~\ref{iso-notAp}) that for a homogeneous tree equipped with the geodesic distance and the canonical flow measure (see Section \ref{QSmappings} for the precise definition), the Jacobian of an isometry is not an $A_\infty(\mu)$ weight. 
However, if we change the metric to the Gromov distance, we can establish a positive result: the Jacobian of a bilipschitz mapping and its reciprocal belong to $L^\infty$ (see Theorem~\ref{bilipAp}). Moreover, every such mapping induces $(1,C)$-quasi-isometries with respect to the geodesic distance on the tree.
It then makes sense to argue that the link between BMO functions and $A_p(\mu)$ weights in the setting of homogeneous trees, endowed with the geodesic distance, is better represented by (a subset of) quasi-isometric bijections, rather than quasisymmetric mappings. We plan to investigate this question in the future. 

\smallskip

Let us discuss some other open problems related to the theory of $A_p$ weights that will be the object of further investigation. It would be interesting to study a suitable version of a Fefferman-Stein inequality adapted to our setting in the spirit of \cite{Om, Om2, NA} and weighted estimates for integral operators, such as Riesz transforms and spectral multipliers of the flow Laplacian, for which unweighted $L^p$ estimates were obtained in \cite{HS, LMSTV, MSTV, MSV}. Another challenging problem is the study of two-weight inequalities in this context (see \cite{ACLM} for a recent contribution in this direction).

It is also worth mentioning that the discrete setting of trees equipped with flow measures has a natural continuous counterpart given by solvable extensions of stratified Lie groups equipped with flow measures (see \cite{HS, DLMV, MOV, V}). It is then natural to define $A_p$ weights in this context and study a characterization of such weights in terms of the boundedness properties of a suitable maximal operator. 

\bigskip

We will use the variable constant convention and write $C$, possibly with subscripts, to denote a positive constant that may vary from place to place and depend on any factor quantified (implicitly or explicitly) before its occurrence but not on factors quantified afterwards. If two positive quantities $A$ and $B$ are comparable (i.e., $\frac1CB \le A \le CB$ for some $C > 0$), we shall write $A \approx B$. We may add a subscript to $\approx$ to trace the dependence of $C$ on a parameter. 
\section{Admissible trapezoids and $A_p$ weights}\label{Sec:2}  

Throughout this section, we consider an infinite tree $T$ with root $\omega_*$ at infinity and we  assume that 
\begin{itemize}
    \item[$i)$]$\mu$ is a flow measure on $T$ (possibly not locally doubling);
       \item[$ii)$] $w$ is a weight on $T$, i.e., a positive function defined on $T$.
\end{itemize}
    Notice that there are no assumptions on the geometry of $T$. Indeed, $(T,d)$ need not be a locally compact space.

We fix a number $\beta \ge 12$ once and for all and  we define for every $s \ge 0$
\begin{align*}
    \lceil s \rceil:=\min\{j \in \mathbb N : j \ge s\}, \qquad \lfloor s\rfloor:=\max\{j \in \mathbb N : j \le s\}.
\end{align*}

\begin{definition}
A {\it trapezoid} with root $x\in T$ and heights $h_1,h_2\in\mathbb N$, $h_1< h_2$, is defined by 
$$R_{h_1}^{h_2}(x):=\{y \in T \ : \ y \le x, h_1 \le d(x,y)<h_2\}. $$
We point out that \begin{align}\label{eq:measR}\mu(R_{h_1}^{h_2}(x))=(h_2-h_1)\mu(x).\end{align} 
A trapezoid $R=R_{h_1}^{h_2}(x)$ is called admissible if $2\le \frac{h_2}{h_1}\le \beta$, and in such case we define its {\it envelope} as the set  \begin{align}\label{envelope}
    \widetilde{R}:=R_{\lceil h_1/\beta\rceil}^{h_2\beta}(x).
\end{align} 
\end{definition}

We denote by $\mathcal{R}$ the family of all admissible trapezoids and singletons. We shall think of a singleton $\{x\}$ as a trapezoid with root $x$ and we say that the envelope of a singleton is the singleton itself. 

 \smallskip

We now recall some useful facts about admissible trapezoids. For detailed proofs, we refer to \cite{LSTV}.
\begin{lemma}\label{lem:LSTV}Let $R_1,R_2 \in \mathcal R$ with roots $x_1$, $x_2 \in T$, such that $R_1 \cap R_2 \ne \emptyset$ and assume that
$\mu(x_1) \ge  \mu(x_2)$. Then,
$R_2 \subset \widetilde{R}_1$.
\end{lemma}
 As a consequence of Lemma \ref{lem:LSTV},  the {\it noncentred} Hardy--Littlewood maximal operator $\mathcal{M}_\mu$ defined on a function $f$ on $T$ by 
\begin{align*}
    \mathcal{M}_\mu f(x)=\sup_{R \in \mathcal{R}\hspace{0.05cm}:\hspace{0.05cm}x\in R} \frac{1}{\mu(R)} \sum_{y \in R} |f(y)| \mu(y) \qquad \forall x \in T,
\end{align*} is bounded from $L^{1}(\mu)$ to $L^{1,\infty}(\mu)$ and on $L^p(\mu)$ for every $p>1$. We remark that this was proved assuming that $\mu$ is locally doubling. A careful inspection of the proof of \cite[Theorem 3.3]{LSTV} reveals that this property is not necessary to prove the weak type $(1,1)$ boundedness of $\mathcal{M}_\mu$. Note that the aforementioned assumption implies a restriction on the geometry of a tree. Specifically, if $T$ admits a locally doubling flow measure, then $T$ has bounded degree (as discussed in \cite[Corollary 2.3]{LSTV}). We emphasize that in the first part of this paper, there are no hypotheses regarding the geometry of the tree; it may not even be locally finite.  


Inspired by the classical results obtained in the Euclidean setting, namely on $\mathbb R^n$ equipped with the Lebesgue measure (see \cite{Mu1}), we now focus on the study of the boundedness of $\mathcal{M}_\mu$ on $L^p(w\mu)$.  In other words, the measure $\mu$ serves as the underlying measure, and our goal is to investigate the boundedness of $\mathcal{M}_\mu$ on the weighted $L^p$ spaces. Clearly, $w\mu$ is a weight, so for notational convenience, for every $p \in [1,\infty)$ we define $L^p_\mu(w):=L^p(w\mu)$ and $L^{p,\infty}_\mu(w):=L^{p,\infty}(w\mu).$ We denote by $\chi_E$ the characteristic function supported on $E \subset T$.  
\\ Assume that $\cM_\mu$ is bounded on $\Lpmuw$ for some $p \in (1,\infty).$ By testing $\cM_\mu$ on $\chi_R w^{-p/p'}$ where $R \in \cR$, it is easy to verify that necessarily  \begin{align}\label{nec cond}
   \sup_{R \in \mathcal{R}} \bigg(\frac{1}{\mu(R)} \sum_{y \in R} w(y)\mu(y) \bigg)^{1/p}  \bigg(\frac{1}{\mu(R)} \sum_{y \in R} w(y)^{-1/(p-1)}\mu(y) \bigg)^{1/p'}< \infty,
\end{align} where $p'$ is such that $1/p+1/p'=1$.
Similarly, if $\cM_\mu$ is bounded from $L^1_\mu(w)$ to $L^{1,\infty}_\mu(w)$, then one can prove that
\begin{align*}
    \cM_\mu(w)(x) \le C w(x), \qquad \forall x \in T.
\end{align*}

This motivates the following definition of $A_p(\mu)$ weights. 
\begin{definition} Let $p \in (1,\infty)$. We say that a weight $w$ is  an $A_p(\mu)$ weight if 
\begin{align}\label{defAp}
    [w]_{A_p(\mu)}:=\sup_{R \in \mathcal{R}} \bigg(\frac{1}{\mu(R)} \sum_{y \in R} w(y)\mu(y) \bigg)  \bigg(\frac{1}{\mu(R)} \sum_{y \in R} w(y)^{-1/(p-1)}\mu(y) \bigg)^{p-1}< \infty.
\end{align} We say that $w$ is an $A_1(\mu)$ weight if 
\begin{align*}
        [w]_{A_1(\mu)}:=\sup_{R \in \mathcal{R}} \bigg(\frac{1}{\mu(R)} \sum_{y \in R} w(y) \mu(y) \bigg)\|w^{-1}\|_{L^\infty(R)}<\infty.
    \end{align*}
\end{definition} For notational convenience we set  $w_\mu(E):=\sum_{y \in E}w(y)\mu(y)$ for every $E \subset T$. Hence, for all $p \in (1,\infty)$ we can write
\begin{align*}
    [w]_{A_p(\mu)}&=\sup_{R \in \mathcal{R}} \frac{w_\mu(R)}{\mu(R)}\bigg(\frac{1}{\mu(R)} \sum_{y \in R} w(y)^{-1/(p-1)}\mu(y) \bigg)^{p-1}, 
    \end{align*} and similarly 
    \begin{align*}
     [w]_{A_1(\mu)}&=\sup_{R \in \mathcal{R}} \frac{w_\mu(R)}{\mu(R)}\|w^{-1}\|_{L^\infty(R)}.
\end{align*}
  The next proposition summarizes some properties of $A_p(\mu)$ weights.
\begin{proposition}\label{prop:Ap} Let $w \in A_p(\mu)$  and $p \in [1,\infty)$. Then, 
\begin{itemize}
    \item[$i)$] $[\lambda w]_{A_p(\mu)}=[w]_{A_p(\mu)}$ for all $\lambda>0$; \\ 
    \item[$ii)$] if $p\in(1,\infty)$, then the function $w^{-1/(p-1)}$ is in $A_{p'}(\mu)$, where $1/p'=1-1/p$, and 
    \begin{align*}
        [w^{-1/(p-1)}]_{A_{p'}(\mu)}=[w]_{A_p(\mu)}^{1/(p-1)};
    \end{align*} 
    \item[$iii)$] $[w]_{A_p(\mu)} \ge 1$ and the equality holds if and only if $w$ is a constant weight; 
    \item[$iv)$] the classes $A_p(\mu)$ are increasing with $p$, specifically
    \begin{align*}
        [w]_{A_q(\mu)} \le [w]_{A_p(\mu)}, 
    \end{align*} whenever $1<p\le q<\infty;$
    \item [$v)$] $\lim_{q \to 1^+} [w]_{A_q(\mu)}=[w]_{A_1(\mu)}$ if $w \in A_1(\mu)$;
    \item[$vi)$] we have the following equivalent characterization of $A_p(\mu)$ weights:
    \begin{align*}
        [w]_{A_p(\mu)}=\sup_{R \in \mathcal{R}} \sup_{|f| \not\equiv 0\hspace{0.1cm}  on \hspace{0.07cm} R} \frac{\bigg(\frac{1}{\mu(R)}\sum_{y \in R} |f(y)| \mu(y)\bigg)^p}{\frac{1}{w_{\mu}(R)}\sum_{y \in R} |f(y)|^p w(y) \mu(y)}.
    \end{align*}
\end{itemize}
\end{proposition}
\begin{proof}
    All of these properties can be proved by using well-known arguments as in \cite[Proposition 9.1.5.]{Gra}; we omit the details. 
\end{proof} Recall that $w :\mathbb Z \to (0,\infty)$ is an $A_p$ weight on $\mathbb Z$ with $p \in (1,\infty)$ if
\begin{align*}
    [w]_{A_p(\mathbb Z)}:=\sup_{{I \ \mathrm{intervals \ in\ }\mathbb Z}} \bigg(\frac{1}{|I|} \sum_{x \in I} w(x)  \bigg)\bigg(\frac{1}{|I|} \sum_{x \in I} w(x)^{-1/(p-1)}  \bigg)^{p-1}<\infty,
\end{align*} where $|\cdot|$ denotes the counting measure on $\mathbb Z$. Similarly, 
$$[w]_{A_1(\mathbb Z)}:=\sup_{I \ \mathrm{intervals \ in\ }\mathbb Z} \bigg(\frac{1}{|I|} \sum_{x \in I} w(x)  \bigg)\|w^{-1}\|_{L^\infty(I)}<\infty.$$
One may inquire whether $A_p(\mu)$ is trivial. In this regard, we now show that it is possible to construct plenty of nontrivial $A_p(\mu)$ weights. \\ 
 We denote by $\mathcal{L}(T)$ the family of weights $w$ on $T$ that only depend on the level, i.e., $w(x)=W(\ell(x))$ for some $W:\mathbb Z \to (0,\infty)$ and for every $x \in T.$  In the next theorem we characterize the $A_p(\mu)$ weights in $\mathcal{L}(T)$. Remarkably, it turns out that such weights do not depend on the particular choice of $\mu$. \begin{theorem}\label{th01}For every $p\in [1,\infty)$, $w \in \mathcal{L}(T)$ is an $A_p(\mu)$ weight on $T$ if and only if $W\in A_p(\mathbb Z)$. Moreover, we have that $[w]_{A_p(\mu)}=[W]_{A_p(\mathbb Z)}.$
\end{theorem}
\begin{proof} Assume that $p\in (1,\infty)$. The case $p=1$ is easier and we omit the details. Observe that for every $R=R_{h_1}^{h_2}(x_0) \in \mathcal{R}$, 
\begin{align*}
   &I_R:=\bigg(\frac{1}{\mu(R)} \sum_{x \in R} w(x) \mu(x) \bigg)\bigg(\frac{1}{\mu(R)} \sum_{x \in R} w(x)^{-1/(p-1)} \mu(x) \bigg)^{p-1}= \\ 
   &= \bigg(\frac{1}{\mu(x_0)(h_2-h_1)} \sum_{\ell=\ell(x_0)-h_2+1}^{\ell(x_0)-h_1} W(\ell)\sum_{x \le x_0, \ell(x)=\ell} \mu(x) \bigg)\times \\ 
   &\times\bigg(\frac{1}{\mu(x_0)(h_2-h_1)} \sum_{\ell=\ell(x_0)-h_2+1}^{\ell(x_0)-h_1} W(\ell)^{-1/(p-1)}\sum_{x \le x_0, \ell(x)=\ell} \mu(x) \bigg)^{p-1} 
\end{align*} and since $\mu$ is a flow measure, $\sum_{x \le x_0, \ell(x)=\ell} \mu(x)=\mu(x_0)$ for every $\ell \le \ell(x_0)$. We conclude that 
\begin{align*}
    I_R=\bigg(\frac{1}{h_2-h_1} \sum_{\ell=\ell(x_0)-h_2+1}^{\ell(x_0)-h_1} W(\ell) \bigg)\bigg(\frac{1}{h_2-h_1} \sum_{\ell=\ell(x_0)-h_2+1}^{\ell(x_0)-h_1} W(\ell)^{-1/(p-1)}\bigg)^{p-1}.
\end{align*} Since $|[\ell(x_0)-h_2+1,\ell(x_0)-h_1]|=h_2-h_1$, this implies that $[w]_{A_p(\mu)} \le [W]_{A_p(\mathbb Z)}$. For the converse inequality, it suffices to prove that given an interval $I=[a,b]\subset \mathbb Z$ there exists an admissible trapezoid $R_{h_1}^{h_2}(x_0)$ (where $h_1,h_2$ and $x_0$ depend on $a$ and $b$) such that 
\begin{align*}
    \ell(x_0)-h_1=b, \ \ \  \ell(x_0)-h_2+1=a\ \ \ \mathrm{and} \ \ \ 2\le \frac{h_2}{h_1}\le \beta. 
\end{align*} Hence, we need to show that there exists $x_0$ such that
$$\begin{cases}
1\le h_1=\ell(x_0)-b \\ 
h_2=\ell(x_0)+1-a\\ 
2\le \frac{\ell(x_0)+1-a}{\ell(x_0)-b} \le \beta.
\end{cases}$$ The latter inequality holds for some $\ell(x_0) \in \mathbb Z$ because $x \mapsto \frac{x+1-a}{x-b}$ is decreasing on the interval $(b,\infty)$ (notice that $-1+a-b<0$), $2\le \frac{b+2-a}{1}$ and $\frac{x+1-a}{x-b}\to 1$ as $x \to \infty$.
\end{proof} 
In a doubling metric measure space $(X,\rho,\nu)$, one can prove that the measure $w\nu$ is doubling (see e.g. \cite[Lemma 4]{Ca}), where $w$ is an $A_p$ weight defined as in \eqref{defAp} but using balls instead of trapezoids. However, this does not hold true in our context due to the nondoubling nature of a flow measure. \\ 
By drawing a parallel with the doubling setting, we aim to show that $w_\mu(\widetilde{R})$ is uniformly bounded by a multiple of $w_\mu(R)$ for every admissible trapezoid $R$ and $w \in A_p(\mu)$. Notice that the envelope of an admissible trapezoid, as defined in \eqref{envelope}, never meets the criteria for admissibility. This fact prevents us from following the classical proof. 
Indeed,  if $\widetilde R$ were admissible, choosing $f=\chi_R$ and $\widetilde R$ in place of $R$ in  Proposition \ref{prop:Ap} $vi)$  we would directly deduce that $w_\mu(\widetilde R) \le C  w_\mu(R)[w]_{A_p(\mu)}$ for every $R \in \cR$.
For this reason, we shall use a different strategy based on the following geometric lemma. 
\begin{lemma}\label{lem:Rtilde} Let $R=R_{h_1}^{h_2}(x) \in \mathcal{R}$ and define 
\begin{align*}
    &R_0:=R_{\lceil h_1/\beta\rceil}^{h_1}(x),\\
    &\overline{R}_0:=R_{1}^{\beta}(x),  \\ 
    &R_1:=R_{\lceil(h_1+h_2)/(2\beta) \rceil}^{\lfloor(h_1+h_2)/2\rfloor}(x), \\
    &\overline{R}_1:=R_{\lfloor \beta/2\rfloor}^{\lfloor\beta/2 \rfloor \beta}(x), \\ 
    &R_2:=R_{\lfloor(h_1+h_2)/2\rfloor}^{\lfloor(h_1+h_2)/2\rfloor \beta}(x),\\ 
    &R_3:=R_{\lceil{(h_1+h_2)/2}\rceil\lfloor \beta/2\rfloor}^{\lceil(h_1+h_2)/2 \rceil\lfloor\beta/2\rfloor\beta}(x).
\end{align*} The following cases arise: 
\begin{itemize}
    \item[i)] if  $h_1 \ge 3$, we have that

\begin{align}\label{coveringtildeR}
    \widetilde{R} \subset R_0 \cup R_1 \cup R_2 \cup R_3.
\end{align}
Moreover, $R_j \in \mathcal{R}$ and $\mu(R)\approx_\beta \mu(R_j)$ for $j=0,1,2,3$ and
\begin{align}\label{numero}
    \mu(R)\approx_{\beta} \mu(R_0 \cap R_1) \approx_{\beta} \mu(R_1 \cap R) \approx_{\beta} \mu(R \cap R_2) \approx_{\beta} \mu(R_3\cap R_2);
\end{align} 
\item[ii)] if  $h_1=1$ and $h_2 \ge 3$ or $h_1=2$, we have that 
\begin{align}\label{second-incl}
    \widetilde R \subset \overline{R}_0\cup \overline{R}_1\cup R_2\cup R_3,
\end{align} $\overline{R}_0, \overline{R}_1, R_2, R_3 \in \cR$ and
\begin{align*}
    \mu(R) \approx_{\beta} \mu(\overline{R}_0 \cap R)  \approx_{\beta} \mu(R \cap R_2) \approx_\beta \mu(\overline{R}_1\cap R_2) \approx_{\beta} \mu(R_3\cap R_2) \approx_{\beta} \mu(R_3);
\end{align*}
\item[iii)]  if $h_1=1$ and $h_2=2$, then  
\begin{align}\label{third-incl}
    \widetilde{R} \subset \overline{R}_0 \cup \overline{R}_1,
\end{align} $\overline{R}_0,\overline{R}_1 \in \cR$ and 
\begin{align*}
     \mu(R) =  \mu(R \cap \overline{R}_0) \approx_{\beta} \mu(\overline{R}_0\cap \overline{R}_1)\approx_{\beta}\mu(\overline{R}_0) \approx_\beta \mu(\overline{R}_1).
\end{align*}
\end{itemize}
\end{lemma}
\begin{proof} We prove $i)$.  Suppose $h_1 \ge 3$. Then, it is clear that $R_0, R_2, R_3 \in \mathcal{R}$. To show that $R_1 \in \cR$ we first observe that $$ \frac{\lfloor\frac{h_1+h_2}{2} \rfloor}{\lceil \frac{h_1+h_2}{2\beta} \rceil} \leq \frac{h_1+h_2}{2}\frac{2\beta}{h_1+h_2}=\beta.$$
On the one hand, if  $h_1+h_2 <2\beta$ we get
\begin{align*}
  2\le  h_1 \le \frac{\lfloor\frac{h_1+h_2}{2} \rfloor}{1},
\end{align*}
on the other hand, if $h_1+h_2 \ge 2\beta$ then
\begin{align*}
 2 \le  \frac{(\beta-1)}{2}=\frac{(\beta-1)(h_1+h_2)}{2(h_1+h_2)}    \le  \frac{\beta(h_1+h_2)-2\beta}{h_1+h_2+2\beta} \le \frac{\lfloor\frac{h_1+h_2}{2} \rfloor}{\lceil \frac{h_1+h_2}{2\beta} \rceil} .
\end{align*}
We now prove \eqref{coveringtildeR}. For this purpose, we first notice that the smaller heights of $R_0$ and $\widetilde R$ coincide by definition, and by the fact that $\beta\geq 12$ and $h_2\leq \beta h_1$, we deduce that
$$
h_1>\frac{h_1}{2}(1/\beta+1)+1>\frac{h_1+h_2}{2\beta}+1>  \lceil(h_1+h_2)/(2\beta)\rceil,
$$
which implies $R_0\cap R_1\neq\emptyset$. Moreover, $R_3$ entirely covers the lower part of $\widetilde R$ since $\lceil(h_1+h_2)/2\rceil\lfloor \beta/2 \rfloor\beta \ge h_2\beta$. 

It remains to prove \eqref{numero}. Formula \eqref{eq:measR} yields
\begin{align*}
    \mu(R \cap R_1)&\approx\mu(R \cap R_2)\approx \mu(x)(h_2-h_1)=\mu(R), \\ 
    \mu(R_3 \cap R_2)&=\mu(x)\big(\beta-\lfloor \beta/2 \rfloor\big)\lfloor(h_1+h_2)/2 \rfloor \approx \mu(R_3) \approx_\beta \mu(R).
\end{align*} 
Moreover, 
$$
  \mu(R_0 \cap R_1)=\mu(x)\big(h_1-\lceil (h_1+h_2)/(2\beta)\rceil \big).
$$
By the admissibility condition $h_2\leq \beta h_1$, the fact that $\beta\geq 12$, and the assumption $h_1 \ge 3$, 
\begin{align}
h_1-\lceil (h_1+h_2)/(2\beta)\rceil\geq h_1-\frac{h_1+h_2}{2\beta}-1\geq h_1\Big(1-\frac{1}{2\beta}-\frac12\Big)-1\geq \frac{h_1}{\beta}. 
\end{align}
This implies that 
$$
  \mu(R_0 \cap R_1)\geq \mu(x) \frac{h_1}{\beta}\approx_\beta\mu(R),
$$
as required to conclude the proof of $i)$. 

The remaining assertions are easier and can be proved similarly, so the details are left to the reader.

\end{proof}

\begin{theorem}\label{th1} Let $w \in A_p(\mu)$ for some $p\in[1,\infty)$. Then, there exists a constant $C_{\beta,p}>0$ such that $w_\mu(\widetilde{R}) \le C_{\beta, p}[w]_{A_p(\mu)}^2w_\mu(R)$ for every $R \in \mathcal{R}$.  
\end{theorem} 

\begin{proof} Let $R=R_{h_1}^{h_2}(x) \in \mathcal{R}$ for some $x \in T$ and $h_1,h_2 \in \mathbb N$. By Lemma \ref{lem:Rtilde}  we have that 
\begin{align}\label{eq:tildeRbound}
    \frac{w_\mu(\widetilde R)}{w_\mu(R)} \le \begin{cases}\displaystyle{\sum_{j=0}^{3} \frac{w_\mu(R_j)}{w_\mu(R)}} &\text{$h_1 \ge 3$,}\\ \displaystyle{\frac{w_\mu(\overline{R}_0)+w_\mu(\overline{R}_1)}{w_\mu(R)}+\sum_{j=2}^{3} \frac{w_\mu(R_j)}{w_\mu(R)}} &\text{$h_1 \le 2, h_2 \ge 3$,} \\ \displaystyle{\sum_{j=0}^{1} \frac{w_\mu(\overline{R}_j)}{w_\mu(R)}} &\text{$h_1=1,h_2=2$.}
    \end{cases}
\end{align}
    Clearly, 
    \begin{align*}
        &\frac{w_\mu(R_0)}{w_\mu(R)}=\frac{w_\mu(R_0)}{w_\mu(R_1)}\frac{w_\mu(R_1)}{w_\mu(R)}, \\ 
        &\frac{w_\mu(R_3)}{w_\mu(R)}=\frac{w_\mu(R_3)}{w_\mu(R_2)}\frac{w_\mu(R_2)}{w_\mu(R)}.
    \end{align*} Thus, for $h_1 \ge 3$, by \eqref{eq:tildeRbound}, \eqref{numero}, and Proposition \ref{prop:Ap} $vi)$,
    \begin{align*}
         \frac{w_\mu(\widetilde R)}{w_\mu(R)} &\le C_{\beta,p}\bigg[\frac{w_\mu(R_0)}{w_\mu(R_1)}\bigg(\frac{1}{\mu(R_0)} \sum_{y \in R_0} \chi_{R_1}(y) \ \mu(y)\bigg)^p\frac{w_\mu(R_1)}{w_\mu(R)}\bigg(\frac{1}{\mu(R_1)} \sum_{y \in R_1} \chi_{R}(y) \ \mu(y)\bigg)^p \\ 
         &+\frac{w_\mu(R_1)}{w_\mu(R)}\bigg(\frac{1}{\mu(R_1)} \sum_{y\in R_1} \chi_{R}(y) \ \mu(y)\bigg)^p+\frac{w_\mu(R_2)}{w_\mu(R)}\bigg(\frac{1}{\mu(R_2)} \sum_{y \in R_2} \chi_{R}(y) \ \mu(y)\bigg)^p \\ 
         &+\frac{w_\mu(R_3)}{w_\mu(R_2)}\bigg(\frac{1}{\mu(R_3)} \sum_{y\in R_3} \chi_{R_2}(y) \ \mu(y)\bigg)^p\frac{w_\mu(R_2)}{w_\mu(R)}\bigg(\frac{1}{\mu(R_2)} \sum_{y \in R_2} \chi_{R}(y) \ \mu(y)\bigg)^p \bigg]\\ 
         &\le C_{\beta,p} ([w]_{A_p(\mu)}+[w]_{A_p(\mu)}^2) \le C_{\beta,p} [w]_{A_p(\mu)}^2.
    \end{align*}One can argue similarly in the remaining cases.
\end{proof} We now prove that $A_p(\mu)$ weights are those for which the maximal operator $\mathcal{M}_\mu$ is bounded on $\Lpmuw$. As a preliminary step, it is useful to introduce an auxiliary maximal operator $\cM_\mu^w$, which is defined by 
\begin{align}
    \cM_\mu^wf(x)=\sup_{R \in \mathcal{R}\hspace{0.05cm}:\hspace{0.05cm}x\in R}\frac{1}{w_\mu(R)}\sum_{y \in R} |f(y)|w(y)\mu(y), \qquad \forall x \in T.
\end{align} The next result is an immediate consequence of Theorem \ref{th1}.
\begin{corollary}\label{cor1}
    Let $w$ be an $A_q(\mu)$ weight for some $q \in [1,\infty)$. Then, $\cM^w_{\mu}$ is bounded from $L^1_\mu(w)$ to $L^{1,\infty}_\mu(w)$ and on $L^p(w)$ for every $p \in (1,\infty]$.
\end{corollary}
\begin{proof} It suffices to prove that $\cM^w_{\mu}$ is of weak type $(1,1)$ and interpolate with the obvious $L^\infty$ bound. \\ 
   Let $f \in L^1_\mu(w)$, fix $\lambda>0$ and set $E_\lambda=\{\cM_\mu^w f>\lambda\}$. We observe that there exists a sequence $\{R_j\}_{j \in J} \subset \cR$ such that $\cup_{j \in J} R_j =E_\lambda$ and 
   \begin{align*}
        w_\mu(R_j) < \frac{1}{\lambda}\sum_{y \in R_j} |f(y)|w(y) \mu(y), \qquad \forall j \in J.
    \end{align*} It follows by Lemma \ref{lem:LSTV} that there exists a subsequence $\{R_{j_k}\}_{k} \subset \{R_j\}_{j \in J}$ of mutually disjoint sets such that $E_\lambda \subset \cup_{k} \widetilde R_{j_k}$. Thus, by Theorem \ref{th1}
    \begin{align*}
        w_\mu(E_\lambda) \le \sum_{k} w_\mu(\widetilde R_{j_k}) \le C_{\beta,q,w} \sum_{k} w_\mu(R_{j_k}) \le \frac{C_{\beta,q,w}}{\lambda} \|f\|_{L^1_\mu(w)},
    \end{align*} as desired. 
\end{proof} The next theorem is an analogue of a result in the spirit of Muckenhoupt  in our setting.
\begin{theorem}\label{Muckenhoupt-Main}
   Let $w\in A_p(\mu)$ for some $p\in(1,\infty)$. Then, $\mathcal{M}_\mu$ is bounded on $\Lpmuw$. Moreover, if $w \in A_1(\mu)$, then  $\mathcal{M}_\mu$ is bounded from $L^1_\mu(w)$ to $L^{1,\infty}_\mu(w)$.
\end{theorem}
\begin{proof} Corollary \ref{cor1} allows us to follow the proof used in the Euclidean setting; we provide  the details for completeness.
Assume that $p \in (1,\infty)$ and let $w$ be an $A_p(\mu)$ weight and $f \in \Lpmuw$. Define $\sigma:=w^{-\frac{1}{p-1}}$. Observe that for every $R \in \mathcal{R}$
    \begin{align}\label{1}
        \frac{1}{\mu(R)}\sum_{y \in R}|f(y)|\mu(y)=\frac{w_\mu(R)^{\frac{1}{p-1}}\sigma_\mu(R)}{\mu(R)^{p/(p-1)}}\bigg(
        \frac{\mu(R)}{w_{\mu}(R)}\bigg(\frac{1}{\sigma_\mu(R)}\sum_{y \in R} |f(y)| \mu(y) \bigg)^{p-1}\bigg)^{\frac{1}{p-1}}.
    \end{align}
    It is clear that 
    \begin{align}\label{2}
        \frac{1}{\sigma_\mu(R)}\sum_{y \in R} |f(y)| \mu(y) \le \cM_\mu^{\sigma} (f\sigma^{-1})(x), \qquad \forall x \in R,
    \end{align} thus 
    \begin{align*}
      \frac{1}{\sigma_\mu(R)}\sum_{y \in R} |f(y)| \mu(y) \le   \inf_{x \in R} \cM_\mu^\sigma (f\sigma^{-1})(x) 
    \end{align*} and
    \begin{align*}
       \mu(R)\bigg( \frac{1}{\sigma_\mu(R)}\sum_{y \in R} |f(y)| \mu(y) \bigg)^{p-1} \le \sum_{y \in R} \bigg[\cM_\mu^{\sigma} (f\sigma^{-1})(y)\bigg]^{p-1}\mu(y).
    \end{align*}
    Moreover,  by definition of $[w]_{A_p(\mu)}$
    \begin{align}\label{3}
        \frac{w_\mu(R)\sigma_\mu(R)^{p-1}}{\mu(R)^{p}} \le [w]_{A_p(\mu)}.
    \end{align} By \eqref{1}, \eqref{2}, and \eqref{3}
    \begin{align*}
        \frac{1}{\mu(R)}\sum_{y \in R} |f(y)| \mu(y) \le [w]_{A_p(\mu)}^{1/(p-1)}\bigg[\cM^w_\mu \bigg(\cM^\sigma_\mu(|f|\sigma^{-1})^{p-1}w^{-1}\bigg)(x)\bigg]^{1/(p-1)}, \qquad \forall x \in R.
    \end{align*}
    We deduce that $\mathcal{M}_\mu f \le [w]_{A_p(\mu)}^{1/(p-1)}\bigg[\cM^w_\mu \bigg(\cM^\sigma_\mu(|f|\sigma^{-1})^{p-1}w^{-1}\bigg)\bigg]^{1/(p-1)}.$ By computing the $L^p_\mu(w)$ norm, we obtain that 
    \begin{align*}
        \|\mathcal{M}_\mu f\|_{L^p_\mu(w)} &\le [w]_{A_p(\mu)}^{1/(p-1)} \|\cM^w_\mu \big(\cM^\sigma_\mu(|f|\sigma^{-1})^{p-1}w^{-1}\big)\|_{L^{p'}_\mu(w)}^{1/(p-1)} \\ 
        &\le  [w]_{A_p(\mu)}^{1/(p-1)} \|\cM^w_\mu\|_{L^{p'}_\mu(w)\to L^{p'}_\mu(w)}^{1/(p-1)}\|\cM^\sigma_\mu(|f|\sigma^{-1})^{p-1}w^{-1}\|_{L^{p'}_\mu(w)}^{1/(p-1)} \\ 
        &=[w]_{A_p(\mu)}^{1/(p-1)} \|\cM^w_\mu\|_{L^{p'}_\mu(w)\to L^{p'}_\mu(w)}^{1/(p-1)}\|\cM^\sigma_\mu(|f|\sigma^{-1})\|_{L^{p}_\mu(\sigma)} \\ 
        &\le [w]_{A_p(\mu)}^{1/(p-1)} \|\cM^w_\mu\|_{L^{p'}_\mu(w)\to L^{p'}_\mu(w)}^{1/(p-1)}\|\cM^\sigma_\mu\|_{L^{p}_\mu(\sigma)\to L^{p}_\mu(\sigma)}\|f\|_{L^{p}_\mu(w)}.
    \end{align*} The conclusion now follows by Corollary \ref{cor1}. \\ 
    If $p=1$ and $w \in A_1(\mu)$, we have that, for every $f \in L^1_\mu(w)$
      \begin{align*}
        \frac{w_\mu(R)}{\mu(R)}\sum_{y \in R} |f(y)| \mu(y)&= \frac{w_\mu(R)}{\mu(R)}\sum_{y \in R} |f(y)| \frac{w(y)}{w(y)}\mu(y) \\ &\le \frac{w_\mu(R)}{\mu(R)}\|w^{-1}\|_{L^\infty(R)}\sum_{y \in R} |f(y)| w(y)\mu(y) \\ 
        &\le [w]_{A_1(\mu)}\sum_{y \in R} |f(y)| w(y)\mu(y).
    \end{align*} This implies that 
     \begin{align*}
        \mathcal M_\mu f(x) \le [w]_{A_1(\mu)} \mathcal{M}_\mu^w f(x), \qquad \forall x \in T.
    \end{align*} Again, an application of Corollary \ref{cor1} concludes the proof.
\end{proof} In summary, we have proved the following version of the Muckenhoupt Theorem.
\begin{corollary}
    Let $w$ be a weight on $T$. Then, 
    \begin{itemize}
        \item[i)] $\cM_\mu$ is bounded on $\Lpmuw$ if and only if $w \in A_p(\mu)$;
        \item [ii)] $\cM_\mu$ is bounded from $L^1_\mu(w)$ to $L^{1,\infty}_\mu(w)$ if and only if $w \in A_1(\mu)$.
    \end{itemize}
     
\end{corollary}
\section{Reverse H\"older inequality and $A_\infty$ weights}\label{secRevHolder}
In this section we assume that $\mu$ is a locally doubling flow measure. Under this additional assumption, which in turn implies uniform boundedness of the (local) geometry of the tree, a Calder\'on-Zygmund theory was developed on $(T,d,\mu)$ using admissible trapezoids instead of metric balls (see \cite[Section 3]{LSTV}). In particular, there exist a constant $C_D\in (0,1)$ and an integer $N$ such that for every admissible trapezoid $R$ there are at most $M$ disjoint $R_i\in\mathcal R$, $i=1,\dots,M$, such that $M \le N$ and 
\begin{align}\label{algdecomposition}
R=\bigcup_{i=1}^M R_i\qquad {\rm{and}}\qquad   \frac{\mu(R_i)}{\mu (R)}\geq C_D\qquad \forall i=1,\dots,M.
\end{align}
The previous fact and a stopping-time argument provide the following useful result (see  \cite[Lemma 3.5]{LSTV}).
\begin{lemma} \label{lem:CZdec}   Let $f$ be a function on $T$, $\lambda > 0$  and $R \in \mathcal{R}$ such that 
$$\frac{1}{\mu(R)} \sum_{y \in R}|f(y) | \mu(y) < \lambda.$$ Then, there
exist a {(possibly empty)} family $\mathcal{F}$  of disjoint admissible trapezoids and a constant $D_{CZ}$ such that for each $E \in \mathcal{F}$ the following
hold: 
\begin{itemize}
    \item [i)] $\frac{1}{\mu(E)} \sum_{y \in E} |f (y)|\mu(y) \ge \lambda$;
 \item[ii)] $\frac{1}{\mu(E)} \sum_{y \in E}
|f(y) | \mu(y) < 
D_{CZ}\lambda$; \item[iii)]  if $x \in R \setminus \cup_{E\in \mathcal F} E,$ then $|f (x)| < \lambda$.
\end{itemize}
\end{lemma} 
We aim to prove a reverse H\"older inequality. We start by recalling the following preliminary lemma which follows by a classical argument (see \cite[Lemma 9.2.1]{Gra} and Proposition \ref{prop:Ap} $vi)$).
\begin{lemma}\label{lem:pre-reverse} Let $w\in A_p(\mu)$ for some $p \in [1,\infty)$. Then, for every $ \xi \in (0,1)$ and $S \subset R \in \mathcal{R}$ such that $\mu(S) \le \xi \mu(R),$ it follows that $w_\mu(S) \le \Big(1-\frac{(1-\xi)^p}{[w]_{A_p(\mu)}}\Big) w_\mu(R).$
\end{lemma}
Now we can prove the reverse H\"older inequality in our setting.
\begin{theorem}[Reverse H\"older inequality]\label{revholder} Let $w \in A_p(\mu)$ for some $p \in [1,\infty)$. Then, there are $C,\varepsilon>0$ such that for every $R \in \mathcal{R}$ 
\begin{align*}
   \bigg( \frac{1}{\mu(R)} \sum_{y \in R} w(y)^{1+\varepsilon}\mu(y) \bigg)^{1/(1+\varepsilon)} \le \frac{C}{\mu(R)} \sum_{y \in R} w(y)\mu(y).
\end{align*}
\end{theorem}
\begin{proof}
    Pick $R \in \mathcal{R}$ and define 
    \begin{align*}
        \gamma_0=\frac{1}{\mu(R)}\sum_{y \in R} w(y) \mu(y).
    \end{align*} We fix $\gamma \in (0,1)$  and we define $\gamma_{k}=(D_{CZ} \gamma^{-1})^k\gamma_0$. We apply Lemma \ref{lem:CZdec} with  $\lambda=\gamma_k$ and $f=w$ and we call $\mathcal{F}_k$ the obtained family of mutually disjoint admissible trapezoids. Then, for every $E \in \mathcal{F}_k$, the following properties are satisfied 
  \begin{itemize}
      \item[$a)$] $\gamma_k<\frac{1}{\mu(E)}\sum_{y \in E} w(y) \mu(y) \le D_{CZ} \gamma_k;$ 
      \item[$b)$] $w \le \gamma_k$ on $R \setminus U_k$ where $U_k:=\cup_{E \in \mathcal{F}_k} E;$ 
      \item[$c)$] each $E' \in \mathcal{F}_{k+1}$ is contained in some $E \in \mathcal{F}_k$.
  \end{itemize} Indeed, $a),b)$ are immediate consequences of Lemma \ref{lem:CZdec} $i),ii)$, and $c)$ follows by recalling the stopping-time nature of the proof for Lemma \ref{lem:CZdec}. Observe that, if $w=\gamma_0$ on $R$, then the reverse H\"older inequality is trivially satisfied. So we can assume that $w$ is not constant on $R$, so that $U_0 \not=\emptyset$.
  Let $k_f=k_f(R,w)$ be the smallest positive integer  such that $U_{k_f}=\emptyset$, 
   which exists because $w \in L^{\infty}(R)$.
    We show that for every $k<k_f$ and $E\in \mathcal{F}_k$,
    \begin{equation}\label{ineq-Fk}
    \mu(E \cap U_{k+1}) < \gamma \mu(E).
    \end{equation}
    This is trivial when $k=k_f-1$, whereas for every $k<k_f-1$,
  \begin{align*}
      D_{CZ} \gamma_k &\ge \frac{1}{\mu(E)}\sum_{y\in E} w(y) \mu(y)\\ & \ge \frac{1}{\mu(E)}\sum_{y\in E \cap U_{k+1}} w(y) \mu(y) \\ &=\frac{1}{\mu(E)}\sum_{\substack{ E' \in \mathcal{F}_{k+1} \\ E'\subset E}} \frac{\mu(E')}{\mu(E')}\sum_{y \in E'} w(y) \mu(y)  \\ 
      &>\frac{\mu(E \cap U_{k+1})}{\mu(E)}\gamma_{k+1} \\ 
      &=\frac{\mu(E \cap U_{k+1})}{\mu(E)}D_{CZ}\gamma^{-1}\gamma_{k}.
  \end{align*} 
  By~\eqref{ineq-Fk} and Lemma \ref{lem:pre-reverse}, we conclude that
  \begin{align*}
      \frac{w_\mu(E \cap U_{k+1})}{w_\mu(E)}\le \eta=1-\frac{(1-\gamma)^p}{[w]_{A_p(\mu)}}.
  \end{align*} Taking the sum over all $E \in \mathcal{F}_k$, 
  \begin{align*}
      w_\mu(U_{k+1}) \le \eta w_\mu(U_k), \qquad \forall k<k_f.
  \end{align*} 
   It is clear that 
  \begin{align*}
      R= (R \setminus U_0) \bigcup (\bigcup_{k=0}^{k_f-1} U_k \setminus U_{k+1}), 
  \end{align*} thus 
  \begin{align*}
      \sum_{y \in R} w(y)^{1+\varepsilon}\mu(y)&=\sum_{y \in R \setminus U_0} w(y)^{\varepsilon}w(y)\mu(y)+\sum_{k=0}^{k_f-1}\sum_{y \in U_k \setminus U_{k+1}} w(y)^{\varepsilon}w(y) \mu(y)\\ &\le \gamma_0^\varepsilon w_\mu(R \setminus U_0)+\sum_{k=0}^{k_f-1}\gamma_{k+1}^\varepsilon w_\mu(U_k)\\ 
      &\le \gamma_0^\varepsilon w_\mu(R \setminus U_0)+\sum_{k=0}^{k_f-1}[(D_{CZ}\gamma^{-1})^{k+1}\gamma_0]^\varepsilon \eta^k w_\mu(U_0) \\ 
      &\le w_\mu(R)\gamma_0^\varepsilon\bigg(1+(D_{CZ}\gamma^{-1})^\varepsilon\sum_{k=0}^{k_f-1} (D_{CZ}\gamma^{-1})^{\varepsilon k} \eta^k\bigg) \\ 
      &\le C\bigg(\frac{1}{\mu(R)} \sum_{y \in R} w(y) \mu(y)\bigg)^{\varepsilon}  \sum_{y \in R} w(y) \mu(y),
  \end{align*} where $C$ does not depend on $R$ provided $(D_{CZ}\gamma^{-1})^{\varepsilon}\eta<1.$ Notice that this is possible since $\eta<1$ and the fact that we can choose $\varepsilon>0$ small enough.
\end{proof} 

A straightforward adaptation of the proof above implies the following variant of  Theorem \ref{revholder}, which will be useful later.
\begin{proposition}\label{rem:revhold} Let $w$ be a weight such that there exist $\xi,\eta \in (0,1)$ satisfying
\begin{align}\label{eq:weakerassumption}
         \mu(S) < \xi \mu(R) \Longrightarrow w_\mu(S) < w_\mu(R),
    \end{align} for every $R \in \mathcal{R}$ and $S \subset R$.  Then, there are  $C,\varepsilon>0$ such that for every $R \in \cR$
    \begin{align*}
   \bigg( \frac{1}{\mu(R)} \sum_{y \in R} w(y)^{1+\varepsilon}\mu(y) \bigg)^{1/(1+\varepsilon)} \le \frac{C}{\mu(R)} \sum_{y \in R} w(y)\mu(y).
\end{align*}
\end{proposition}
 The following are easy consequences of Theorem \ref{revholder}.
\begin{corollary} The following hold: 
\begin{itemize} 
    \item[i)] 
let $w \in A_p(\mu)$ for some $p \in [1,\infty)$. Then, there exists $\varepsilon>0$ such that $w^{1+\varepsilon} \in A_p(\mu)$; 
\item[ii)] if $p \in (1,\infty),$ then  $$A_p(\mu)=\bigcup_{s \in (1,p)} A_s(\mu).$$
\end{itemize}
\end{corollary}
\begin{proof} The proof of these facts is standard, but we give all the details for completeness.
    Assume $w\in A_1(\mu)$. Then, for all $R \in \cR$, by Theorem \ref{revholder} 
    \begin{align*}
        \frac{1}{\mu(R)}\sum_{y \in R} w(y)^{1+\varepsilon} \mu(y) \le C[w]_{A_1(\mu)}^{1+\varepsilon} w(x)^{1+\varepsilon}, \qquad \forall x \in R,
    \end{align*} namely, $w^{1+\varepsilon} \in A_1(\mu)$. \\ If $p>1$, we apply again Theorem \ref{revholder} to the weights $w \in A_p(\mu)$ and $w^{-1/(p-1)} \in A_{p'}(\mu)$. Thus, we have that there exist $\varepsilon_1,\varepsilon_2 \in (0,1)$ such that
     \begin{align}
         \label{stella1}&\bigg(\frac{1}{\mu(R)}\sum_{y \in R} w(y)^{1+\varepsilon_1} \mu(y) \bigg)^{1/(1+\varepsilon_1)}\le \frac{C}{\mu(R)}\sum_{y \in R} w(y) \mu(y), \\ 
        \label{stella2} &\bigg(\frac{1}{\mu(R)}\sum_{y \in R} w(y)^{-(1+\varepsilon_2)/(p-1)} \mu(y) \bigg)^{1/(1+\varepsilon_2)}\le \frac{C}{\mu(R)}\sum_{y \in R} w(y)^{-1/(p-1)} \mu(y). 
    \end{align} Set $\varepsilon:=\min\{\varepsilon_1,\varepsilon_2\}$ and observe that, by H\"older's inequality, \eqref{stella1} and \eqref{stella2} are both satisfied with $\varepsilon$ in place of $\varepsilon_1$ and $\varepsilon_2$. Hence, 
    \begin{align*}
        &\bigg(\frac{1}{\mu(R)}\sum_{y \in R} w(y)^{1+\varepsilon} \mu(y) \bigg)\bigg(\frac{1}{\mu(R)}\sum_{y \in R} w(y)^{-(1+\varepsilon)/(p-1)} \mu(y) \bigg)^{p-1} \\ &\le  C_{p,\varepsilon} \bigg[\frac{1}{\mu(R)}\sum_{y \in R} w(y) \mu(y) \bigg( \frac{C}{\mu(R)}\sum_{y \in R} w(y)^{-1/(p-1)} \mu(y) \bigg)^{p-1}\bigg]^{1+\varepsilon},
    \end{align*}
    
   which implies that $w^{1+\varepsilon} \in A_p(\mu)$ and \begin{align}\label{w1+epsilon}
        [w^{1+\varepsilon}]_{A_p(\mu)} \le C_{p,\varepsilon}[w]_{A_p(\mu)}^{1+\varepsilon}.
    \end{align} This proves $i)$.  \\ Next, given $w \in A_p(\mu)$ and $\varepsilon \in (0,1)$ as in the proof of $i)$, we set $\delta=\frac{1}{1+\varepsilon}$ and $s=p\delta+1-\delta=\frac{p+\varepsilon}{1+\varepsilon}$. We claim that 
    \begin{align}\label{claimdelta}
        [w^\delta]_{A_s(\mu)} \le [w]_{A_p(\mu)}^\delta.
    \end{align} This easily follows by definition of $A_p(\mu)$. Indeed, 
    \begin{align*}
        [w^\delta]_{A_s(\mu)}&= \sup_{R \in \mathcal{R}} \bigg(\frac{1}{\mu(R)}\sum_{y \in R} w(y)^{\delta} \mu(y) \bigg)\bigg(\frac{1}{\mu(R)}\sum_{y \in R} w(y)^{-\delta/(s-1)} \mu(y) \bigg)^{s-1} \\ 
        &\le \sup_{R \in \mathcal{R}} \bigg(\frac{1}{\mu(R)}\sum_{y \in R} w(y) \mu(y) \bigg)^{\delta}\bigg(\frac{1}{\mu(R)}\sum_{y \in R} w(y)^{-1/(p-1)} \mu(y) \bigg)^{\delta(p-1)}\\ &=[w]_{A_p(\mu)}^\delta, 
    \end{align*} where in the above inequality we have used H\"older's inequality. By \eqref{claimdelta} applied to the weight $w^{1/\delta} \in A_p(\mu)$ instead of $w$ and  \eqref{w1+epsilon}, we conclude that 
    \begin{align*}
        [w]_{A_s(\mu)}\le [w^{1+\varepsilon}]_{A_p(\mu)}^{1/(1+\varepsilon)}\le C_{p,\varepsilon}[w]_{A_p(\mu)}.
    \end{align*} Since $1<s<p$, the desired result follows.
\end{proof}
We underline that if $w \in A_p(\mu)$, then $w \mu$ is not, in general, a flow measure or a doubling measure, so a variant of Lemma \ref{lem:CZdec} for the measure $w\mu$ is needed. 
%
To do so, we will work with weights satisfying the following assumption.

\smallskip

 {\bf Assumption $1$}: $w$ is a weight and $\eta \in (0,1)$ is such that for every $R \in \mathcal{R}$ and $S \subset R$
    \begin{align}\label{hp:CZbis}
        \mu(S) \le (1-C_{D})\mu(R) \Longrightarrow w_\mu(S) \le \eta w_\mu(R),
    \end{align}  
    where $C_D$ is as in \eqref{algdecomposition}. \\  Observe that Assumption $1$ implies the following:  for some $\alpha \in (0,1)$ there exists $\beta \in (0,1)$ such that for every $R \in \mathcal{R}$ and $S \subset R$
    \begin{align}\label{assumrevw-1}
        w_\mu(S) < \alpha w_\mu(R) \Longrightarrow \mu(S) < \beta \mu(R).
    \end{align} Indeed, if $w_\mu(S)>\eta w_\mu(R)$, then by Assumption 1 necessarily $\mu(S)>(1-C_{D}) \mu(R).$ Since $w_\mu(S)>\eta w_\mu(R)$ if and only if $w_\mu(R \setminus S) < (1-\eta)w_\mu(R)$ and an analogous inequality follows for $\mu$, we have that Assumption $1$ implies 
    \begin{align*}
        w_\mu(S) < (1-\eta) w_\mu(R) \Longrightarrow \mu(S) < C_{D}\mu(R).
    \end{align*}
Under Assumption 1, we can prove the following variant of Lemma \ref{lem:CZdec} for the measure $w\mu$.
\begin{proposition}\label{newCZdec}
     For every function $f$ on $T$, $\lambda > 0$  and $R \in \mathcal{R}$ such that 
$\frac{1}{w_\mu(R)} \sum_{y \in R}
|f(y) | w(y)\mu(y) < \lambda$ there
exist a ({possibly empty}) family $\mathcal{F}$ of disjoint admissible trapezoids and a constant $D_{CZ}'$ such that for each $E \in \mathcal{F}$ the following 
hold: 
\begin{itemize}
    \item [i)] $\frac{1}{w_\mu(E)} \sum_{y \in E} |f (y)|w(y)\mu(y) \ge \lambda$;
 \item[ii)] $\frac{1}{w_\mu(E)} \sum_{y \in E}
|f(y) | w(y)\mu(y) < 
D_{CZ}'\lambda$; \item[iii)]  if $x \in R \setminus \cup_{E\in \mathcal F} E,$ then $|f (x)| < \lambda$.
\end{itemize}
\end{proposition}
\begin{proof} 
    By replacing $S$ with $R\setminus S$ in \eqref{hp:CZbis}, it is clear that Assumption 1 is equivalent to the following:  there exists $\delta \in (0,1)$ such that 
    \begin{align}\label{hp:CZbisequiv}
        \mu(S) \ge C_{D}\mu(R) \Longrightarrow w_\mu(S) \ge \delta w_\mu(R).
    \end{align} 
    Let $R_i \in \mathcal{R}$ be as in \eqref{algdecomposition}. By  \eqref{algdecomposition} and \eqref{hp:CZbisequiv}, it follows that 
    \begin{align*}
        \frac{w_\mu(R_i)}{w_\mu(R)} \ge \delta,
    \end{align*} where $\delta$ is a constant independent of $R \in \cR$. Then, one can follow verbatim the proof of \cite[Lemma 3.5]{LSTV} to conclude the proof. 
\end{proof} The next result is a consequence of the above proposition.
\begin{corollary}\label{revholdw-1} If $w$ is a weight which satisfies Assumption 1, then there exist $\varepsilon,C>0$ such that the following reverse H\"older inequality holds 
    \begin{align*}
        \bigg(\frac{1}{w_\mu(R)}\sum_{y \in R} w^{-1-\varepsilon}(y) w(y)\mu(y) \bigg)^{1/(1+\varepsilon)} \le \frac{C}{w_\mu(R)}\sum_{y \in R} \mu(y), \qquad \forall R \in \cR.
    \end{align*}
\end{corollary} 
\begin{proof}
    The proof follows verbatim the proof of Theorem \ref{revholder} replacing $\mu$ with $w \mu$ and applying Proposition \ref{newCZdec} instead of Lemma \ref{lem:CZdec} and using \eqref{assumrevw-1} as explained in Proposition \ref{rem:revhold}. 
\end{proof} We now introduce the class of $A_\infty(\mu)$ weights.
\begin{definition} We define $A_\infty(\mu)=\bigcup_{p>1}A_{p}(\mu)$ and we set 
\begin{align*}
    [w]_{A_\infty(\mu)}=\sup_{R \in \mathcal{R}} \frac{1}{\mu(R)} \sum_{y \in \mathcal{R}} w(y)\mu(y) \exp{\bigg[\frac{1}{\mu(R)}\sum_{y \in R}\log (w(y)^{-1})\mu(y)\bigg]}.
\end{align*}
\end{definition} 
We  establish below the equivalence between $A_\infty(\mu)$ and the weights satisfying $[w]_{A_\infty(\mu)}<\infty$. Our main result is based on Corollary \ref{revholdw-1}. We refer to \cite{GCRdF, Mu2, CF} for related results in the Euclidean setting and to \cite{OP} for an example in a nondoubling setting. 
\begin{theorem}\label{thAinf} Suppose that $w$ is a weight on $T$. The following are equivalent:
\begin{itemize}
    \item[i)] $w \in A_\infty(\mu)$;
    \item [ii)] $[w]_{A_\infty(\mu)}<\infty;$
    \item [iii)] there is $\gamma_0 \in (0,1)$ such that for all $\gamma \in (0,\gamma_0)$ there exists $\delta=\delta(\gamma) \in (0,1)$ such that for all $R \in \mathcal{R}$
    \begin{align*}
        \mu\Big(\Big\{x \in R : w(x) \le \frac{\gamma}{\mu(R)}\sum_{y \in R} w(y) \mu(y)\Big\}\Big)\le \delta \mu(R).
    \end{align*} Furthermore,  for every $\varepsilon>0$ we can choose $\gamma \in (0,\gamma_0)$ such that $\delta(\gamma) <\varepsilon$;
    \item [iv)] for every $\xi \in (0,1)$ there is $\eta=\eta(\xi) \in (0,1)$ such that for all $R \in \mathcal{R}$ and $S \subset R$ $$\mu(S) \le \xi \mu(R) \Longrightarrow w_\mu(S) \le \eta w_{\mu}(R).$$ 
\end{itemize}
\end{theorem}
\begin{proof} 
 Assume that $w \in A_\infty(\mu)$. Then, since the classes $A_p(\mu)$ are increasing with $p$,  $w \in A_p(\mu)$ for every $p \ge p_0$ for some $p_0 \in (1,\infty).$ Then, 
 \begin{align*}
     \frac{1}{\mu(R)}\sum_{y \in R} w(y) \mu(y) \bigg(\frac{1}{\mu(R)}\sum_{y \in R}w(y)^{-1/(p-1)}\mu(y)\bigg)^{p-1} \le C <\infty,
 \end{align*} for every $p \ge p_0$ and we conclude passing to the limit as $p \to \infty$.
 \\ We now prove that $ii)$ implies $iii)$. Indeed,   assume that $[w]_{A_\infty(\mu)}<\infty$ and let $R \in \mathcal{R}$. We can suppose without loss of generality that $\sum_{y \in R}\log w(y) \mu(y)=0$ (otherwise we can divide $w$ by a suitable constant) so that \begin{align}\label{eq:winf}
     \frac{1}{\mu(R)}\sum_{y \in R}w(y) \mu(y) \le [w]_{A_\infty(\mu)}.
 \end{align} It follows that, given $\gamma>0$ small enough to be chosen later
   \begin{align*}
       \mu\Big(\Big\{x \in R : w(x) \le \frac{\gamma}{\mu(R)} &\sum_{y \in R} w(y) \mu(y)\Big\}\Big) \\ &\le \mu(\{x \in R : w(x) \le \gamma[w]_{A_\infty(\mu)}\}) \\ &=\mu(\{x \in R : \log(1+w(x)^{-1}) \ge \log(1+(\gamma[w]_{A_\infty(\mu)})^{-1})\})\\&\le \frac{1}{\log(1+(\gamma[w]_{A_\infty(\mu)})^{-1})}\sum_{y \in R} \log(1+w(y)) \mu(y)  \\ 
        &=:I, 
       \end{align*} where in the second inequality we have used Chebychev's inequality and the fact that $\sum_{y \in R} \log w(y) \mu(y)=0.$ Next, by multiplying and dividing  $I$ by $\mu(R)$, an application of Jensen's inequality yields
       \begin{align*}
       I&\le \frac{\mu(R)}{\log(1+(\gamma[w]_{A_\infty(\mu)})^{-1})}\log\bigg(\frac{1}{\mu(R)}\sum_{y \in R} (1+w(y)) \mu(y) \bigg) \\&\le \frac{\mu(R)\log(1+[w]_{A_\infty(\mu)})}{\log(1+(\gamma[w]_{A_\infty(\mu)})^{-1})}\\ 
       &=\delta(\gamma) \mu(R),
   \end{align*} where in the last inequality we used \eqref{eq:winf} and we chose $\gamma_0=[w]_{A_\infty(\mu)}^{-2}$, $\gamma \in (0,\gamma_0)$, and $$\delta(\gamma):=\frac{\log(1+[w]_{A_\infty(\mu)})}{\log(1+(\gamma[w]_{A_\infty(\mu)})^{-1})}<1.$$ Observe that when $\gamma \to 0^+$ we have that $\delta(\gamma) \to 0$. \\ 
   Next, assume that $iii)$ holds. Pick $R \in \mathcal{R}$, fix $\xi \in (0,1)$ and suppose that $w_\mu(S)>\eta w_\mu(R)$ for some $\eta \in (0,1)$ to be chosen. Our goal is to show that necessarily $\mu(S)>\xi \mu(R)$. Indeed, define $E=R \setminus S$ and observe that $w_\mu(E)<(1-\eta)w_\mu(R)$. For a given $\gamma \in (0,\gamma_0)$, with $\gamma_0$ as in $iii)$, we set 
   \begin{align*}
       &E_1=\{x \in E \ : \ w(x) >\frac{\gamma}{\mu(R)} \sum_{y \in R} w(y) \mu(y) \} \\ 
       &E_2=\{x \in E \ : \ w(x) \le \frac{\gamma}{\mu(R)} \sum_{y \in R} w(y) \mu(y) \}.
   \end{align*} By $iii)$, $\mu(E_2) \le \delta w_\mu(R).$ Moreover, 
   \begin{align*}
       \mu(E_1) \le \frac{\mu(R)w_\mu(E)}{\gamma w_\mu(R)} < \frac{\mu(R)(1-\eta)}{\gamma},
   \end{align*} by Chebyshev's inequality. Hence we conclude that
   \begin{align*}
       \mu(E) < \bigg(\delta +\frac{1-\eta}{\gamma}\bigg)\mu(R).
   \end{align*} We now choose $\gamma \in (0,\gamma_0)$ small enough such that $1-\delta(\gamma)>\xi$ (this is possible because $\delta(\gamma) \to 0$  as $\gamma \to 0^+)$. Observe that $1-\delta>\xi$ if and only if $1-\xi>\delta$ thus  we choose $\eta \in (0,1)$ such that $\delta+\frac{1-\eta}{\gamma}=1-\xi$. Such a $\eta$ exists because $\delta<1-\xi$ and $\delta+1/\gamma>1>1-\xi$. 
   Summarizing, we have proved that $\mu(E) < (1-\xi) \mu(R),$ namely, that $\mu(S)>\xi \mu(R)$ from which follows $iv)$. \\ 
 Next, $iv)$ implies Assumption 1, thus, by Corollary \ref{revholdw-1} we get 
   \begin{align*}
      \bigg(\frac{1}{w_\mu(R)} \sum_{y \in R} w^{-(1+\varepsilon)} w(y)\mu(y)\bigg)^{1/(1+\varepsilon)} \le C \frac{\mu(R)}{w_\mu(R)}, \qquad\forall R \in \cR,
   \end{align*}
    which is equivalent to 
    \begin{align*}
       \frac{1}{\mu(R)} w_{\mu}(R)^{\varepsilon/(1+\varepsilon)} \bigg( \sum_{y \in R} w^{-\varepsilon}(y) \mu(y)\bigg)^{1/(1+\varepsilon)} \le C, \qquad \forall R \in \cR,
    \end{align*} that is the $A_p(\mu)$ condition for $w$ when $\varepsilon=1/(p-1)$. This implies $i)$ and concludes the proof.
\end{proof} 
In \cite[Section 4]{LSTV}, the space $BMO(\mu)$ is defined by  $$BMO(\mu):=\Big\{f : T \to \mathbb R \ : \ \sup_{R \in \mathcal{R}} \frac{1}{\mu(R)}\sum_{y \in R} |f(y)-f_R|\mu(y)<\infty\Big\}.$$ This space exhibits several good properties. For instance, it is isomorphic to the dual of a suitable Hardy space and it also interpolates with $L^p(\mu)$, $p\in(1,\infty)$ (see \cite[Section 4]{LSTV}). \\ 
In the Euclidean setting, it is known that $BMO$ functions coincide with (multiples) of logarithms of $A_p$ weights (see e.g. \cite[Corollary 2.19]{GCRdF}). In the next proposition, we prove that the same phenomenon occurs in our setting.
\begin{proposition}
Let $w \in A_p(\mu)$ for some $p\in(1,\infty)$ and define $g=\log w$. Then $g \in BMO(\mu)$. Conversely, if $f \in BMO(\mu)$ then  $f=\lambda \log w$ for some $w \in A_p(\mu)$ and $\lambda$ small enough. In other words, $BMO(\mu)=\{\lambda \log w, w \in A_\infty(\mu), \lambda \in \mathbb R\}.$
\end{proposition}
\begin{proof} It suffices to prove the first implication when $w \in A_2(\mu)$. Indeed, if $p <2$ we have that $w \in A_2(\mu)$ and if $p>2$ then  $w^{-1/(p-1)} \in A_{p'}(\mu) \subset A_2(\mu)$ and replacing $w$ by $w^{-1/(p-1)}$ we obtain that $-(p-1)\log w \in BMO(\mu)$. \\ 
For every $R \in \mathcal{R}$ define $$R^+=R \cap \bigg\{x \in R \ : \ g(x)-g_R \ge 0 \bigg\}, \ \ R^-=R \setminus R^+, $$  where $g=\log w$ and for every function $f$ on $T$ we define  $f_R=\frac{1}{\mu(R)}\sum_{y \in R} f(y) \mu(y)$. By Jensen's inequality, we have that 
\begin{align*}
    \exp\bigg[{\frac{1}{\mu(R)}\sum_{y \in R} \chi_{R^+}(y) (g(y)-g_R) \mu(y) }\bigg] \le \frac{1}{\mu(R)} \bigg(\sum_{y \in R^+}(w(y)e^{-g_R})\mu(y)+\sum_{y \in R^-}  \mu(y)\bigg).
\end{align*}  Moreover, since $x \mapsto e^{-x}$ is convex, again by Jensen's inequality $e^{-g_R} \le (w^{-1})_R$ and  $$1=\bigg(\frac{1}{\mu(R)}\sum_{y \in R}w^{-1/2}(y)w^{1/2}(y) \mu(y)\bigg)^2 \le w_R (w^{-1})_R$$  by Cauchy-Schwarz inequality. It follows that 
\begin{align*}
    \frac{1}{\mu(R)}\sum_{y \in R} \chi_{R^+}(y) (g(y)-g_R) \mu(y) \le \log [w_R (w^{-1})_R].
\end{align*} A similar argument shows that 
\begin{align*}
      \exp\bigg[{\frac{1}{\mu(R)}\sum_{y \in R} \chi_{R^-}(y) (g_R-g(y)) \mu(y) }\bigg] &\le \frac{1}{\mu(R)} \bigg(\sum_{y \in R^-}(e^{g_R}w(y)^{-1})\mu(y)+\sum_{y \in R^+}  \mu(y)\bigg) \\ 
      &\le w_R (w^{-1})_R.
\end{align*}
Therefore, 
\begin{align*}
    \frac{1}{\mu(R)} \sum_{y \in R} |\log w(y)-(\log w)_R| \mu(y) \le \log [w_R (w^{-1})_R] \le \log [w]_{A_2(\mu)}.
\end{align*} This shows that $g \in BMO(\mu)$.  \\ Next, pick $f \in BMO(\mu)$ and fix $p \in (1,\infty)$. Define $\psi=e^{\lambda f}$ for some $\lambda=\lambda(p)>0$ to be chosen. Our goal is to show that $\psi \in A_p(\mu)$. Indeed, 
\begin{align*}
    &\frac{1}{\mu(R)}\sum_{y \in R} \psi(y) \mu(y) \bigg( \frac{1}{\mu(R)}\sum_{y \in R} \psi(y)^{-1/(p-1)} \mu(y)\bigg)^{p-1} \\ &=\frac{1}{\mu(R)}\sum_{y \in R} e^{(f(y)-f_R)\lambda} \mu(y) \bigg( \frac{1}{\mu(R)}\sum_{y \in R}e^{-\lambda(f(y)-f_R)/(p-1)} \mu(y)\bigg)^{p-1} \\ 
    &\le \bigg( \frac{1}{\mu(R)}\sum_{y \in R}e^{\eta|(f(y)-f_R|} \mu(y)\bigg)^{p},
\end{align*} where $\eta:=\lambda\max\{1,1/(p-1)\}$. The last expression is uniformly bounded by a constant when $\lambda$ is small enough by  John-Nirenberg's inequality, see \cite[Proposition 4.2. (i)]{LSTV}.
\end{proof} 
\section{$A_\infty(\mu)$ and quasisymmetric mappings}\label{QSmappings}

In this section, we assume that $T=\mathbb T_q$ is the homogeneous tree of order $q+1$, namely the tree such that every vertex has exactly $q+1$ neighbours. The associated canonical flow measure is $\mu(\cdot)=q^{\ell(\cdot)}$ where $\ell$ is the level function on $T$ defined in \eqref{def:lev}. Let $f$ be a bijection from $T$ onto itself and $J_f$ be the weight on $T$ defined by  $$J_f(x)=\frac{\mu(f(x))}{\mu(x)}, \qquad \forall x \in T.$$ 
Observe that, for every finite $E \subset T$, since $f$ is a bijection, we have 
\begin{align*}
     (J_f)_\mu(E):=\sum_{y \in E} J_f(y) \mu(y)=\sum_{y \in E} \mu(f(y)) =\mu(f(E)),
\end{align*}  so $J_f$ can be thought of as the discrete version of the {\it Jacobian} of $f$. 
We remark that, by Theorem \ref{thAinf}, $J_f \in A_\infty(\mu)$ if and only if for every $\xi \in (0,1)$ there exists $\eta \in (0,1)$ such that for every $R \in \mathcal{R}$ and $S \subset R$
\begin{align}\label{remark_J}
    \frac{\mu(S)}{\mu(R)} \le \xi \Longrightarrow  \frac{\mu(f(S))}{\mu(f(R))}=\frac{(J_f)_\mu(S)}{(J_f)_\mu(R)} \le \eta.
\end{align}

 In the next proposition, we shall prove that isometries with respect to the geodesic distance do not generally have Jacobian in $A_\infty(\mu)$. We refer to \cite[Chapter 1]{FTN} for more information about the action of the group of isometries on the homogeneous tree.
\begin{proposition}\label{iso-notAp} Let $f$ be an isometry on $T$ with respect to $d$ that exchanges $\omega_*$ with another boundary point $\omega_- \in \partial T$. Then, $J_f$ does not belong to $A_\infty(\mu)$.
\end{proposition}
\begin{proof}
    Without loss of generality assume that $o \in (\omega_-, \omega_*) $ and that $f(o)=o$. Let  $\{x_n\}_{n \in \mathbb Z}$ denote an enumeration of the infinite geodesic $(\omega_-,\omega_*)$ such that $\ell(x_n)=n.$ Then, we have that $f(x_n)=x_{-n}.$ Let $\{R_n\}_{n \in \mathbb N}\subset \mathcal{R}$ be defined by $R_n=R_{n}^{2n}(o)$ and set  
    $E_n:=\{x \le x_{-n}\}\cap R_n$. Then, 
    \begin{align*}
        \frac{\mu(E_n)}{\mu(R_n)}=q^{-n}.
    \end{align*} It suffices to disprove \eqref{remark_J}  showing that
    \begin{align}\label{limit}
         \frac{\mu(f(E_n))}{\mu(f(R_n))} \to 1 \quad \mathrm{as}\ n \to \infty.
    \end{align} Notice that  $f(E_n)$ contains $f(x_{-2n+1})=x_{2n-1}$, thus $$\mu(f(E_n)) \ge \mu(x_{2n-1})=q^{2n-1}.$$ Since $f$ is an isometry, $\max_{x \in R_n \setminus E_n} d(x_{-n+1},x) \le 3n-1$ and $f(x_{-n+1})=x_{n-1}$, one can see that $$f(R_n\setminus E_n) \subset \{y \ \in T \ : y \le x_{n-1}, d(x_{n-1},y) \le 3n-1\}.$$ Hence $\mu(f(R_n\setminus E_n)) \le 3n q^{n-1}$ and 
    \begin{align*}
         &\frac{\mu(f(R_n\setminus E_n))}{\mu(f(E_n))} \le \frac{3n}{q^{n}}\to 0, \qquad \mathrm{as} \ n \to \infty.
         \end{align*} It follows that
         \begin{align*}
       \frac{\mu(f(E_n))}{\mu(f(R_n))}= \frac{\mu(f(E_n))}{\mu(f(E_n))+\mu(f(R_n\setminus E_n))} \to 1, \qquad \mathrm{as}\ n \to \infty,
    \end{align*}
 and \eqref{limit} is proved.
\end{proof} Intuitively, the problem with these isometries is that they do not consistently preserve the order relation $\le$ between vertices, because the relationship between their images under the isometry is not guaranteed in general. \\ 
 We now introduce a different distance on $T$ which is also natural when we deal with admissible trapezoids. We define the {\it Gromov} distance $\rho : T \times T \to [0,\infty)$ by 
\begin{align*}
    \rho(x,y)=\begin{cases}e^{\ell(x \wedge y)} &x\ne y, \\ 0 &x=y,
    \end{cases}
\end{align*} where $x\wedge y$ denotes the {\it confluent} of $x,y \in T$, namely, the vertex of minimum level which is above both $x$ and $y$. It is easy to verify that $\rho$ is a metric. 

In the next proposition, we study some properties of isometries on $T$ with respect to the Gromov distance $\rho$.
\begin{proposition}\label{isorho}
    Let $f$  be an isometry on $T$ with respect to $\rho$. Then, $f$ is a level and order-preserving onto map, namely, \begin{itemize}
        \item[i)] $\ell(x)=\ell(f(x)),$ for every $x \in T;$
        \item[ii)] $x \le y \iff f(x) \le f(y),$ for every $x,y \in T.$
    \end{itemize}
\end{proposition}
\begin{proof} Let $f$ be an isometry with respect to $\rho$. 
Notice that it suffices to show that $\ell(x)=\ell(f(x))$ for every $x$ in $T$ and \begin{align}\label{claim2}
    x=p(y) \Longrightarrow f(x)=p(f(y)),
\end{align} or in other words $f \circ p = p \circ f,$ where $p(x)$ denotes the predecessor of $x \in T$. 
Fix a point $x_0 \in T$ and let $\{x_j\}_{j=1}^q$ be an enumeration of $s(x_0)$. It easy to see that 
\begin{align*}
    e^{\ell(x_0)}=\rho(x_j,x_k), \qquad \forall j,k=0,1,...,q \ \mathrm{and} \  j \ne k.
\end{align*} Since $f$ is an isometry, it follows that 
\begin{align*}
e^{\ell(x_0)}=\rho(f(x_j),f(x_k))=e^{\ell(f(x_j)\wedge f(x_k))}, \qquad \forall j,k=0,1,...,q \ \mathrm{such \ that} \  j \ne k,
\end{align*} which in turn implies 
\begin{align}\label{eq:confluent}
  \ell(x_0)=\ell(f(x_j) \wedge f(x_k)) \qquad  \forall j,k=0,1,...q, k \ne j.
\end{align}
Clearly, for every $j=1,...,q$ we have that \begin{align*}
    &f(x_0) \le f(x_j)\wedge f(x_0), 
\end{align*} thus, choosing $k=0$ in \eqref{eq:confluent} we get that $z:=f(x_j) \wedge f(x_0)$ does not depend on $j=1,...,q$. Moreover, we observe that for every  $j,k=0,1,...,q$ \begin{align*}
    &f(x_j) \le z \\ &f(x_j) \le f(x_j)\wedge f(x_k),
\end{align*} from which we conclude that $z \le f(x_j)\wedge f(x_k)$ or $f(x_j)\wedge f(x_k)\le z$. Again by \eqref{eq:confluent} we get $z=f(x_j)\wedge f(x_k)$ for every $j,k=0,1,...,q$ such that $j\ne k$. 
Hence, we have found $q+1$ vertices, namely $\{f(x_j)\}_{j=0}^{q}$, that are below $z$ and whose mutual confluent is $z$. Since $T=\mathbb T_q$, the only possibility is that $f(x_{j_0})=z$ for some $j_0=0,...,q$. It is immediate to see that $j_0=0$. Indeed, if $j_0\ge 1$, then by picking any $v \le x_{j_0}$ with $v \ne x_{j_0}$, by \eqref{eq:confluent} we see that $$e^{\ell(f(x_{j_0}))}=e^{\ell(x_0)}>e^{\ell(x_0)-1}=\rho(x_{j_0},v)=\rho(f(x_{j_0}),f(v)) \ge e^{\ell(f(x_{j_0}))},$$ a contradiction. Thus, $j_0=0$ and consequently $z=f(x_0)$ and $\ell(x_0)=\ell(f(x_0))$. By the arbitrariness of $x_0$ and the fact that $f(x_j) \le f(x_0)$ for every $j=1,...,q$ we have \eqref{claim2}.

\end{proof}
It is clear that if $f$ is an isometry with respect to $\rho$  then $J_f=1$, because it is level-preserving. Thus $J_f$ is trivially in $A_p(\mu)$ for every $p \in [1,\infty]$. Next, we  show that if $f$ is a bilipschitz map with respect to $\rho$, i.e., a map from $T$ onto itself such that
\begin{align}\label{bi-lip}
  e^{-C}\rho(f(x),f(y))  \le \rho(x,y) \le e^C \rho(f(x),f(y)), \qquad \forall x,y \in T,
  \end{align}
  then $J_f$ belongs to $A_p(\mu)$ for every $p \in [1,\infty]$.
\begin{proposition}\label{bilipAp}
    Suppose that $f$ is bilipschitz on $T$ with respect to $\rho$. Then, 
\begin{align*}
    -C\le \ell(x)-\ell(f(x)) \le C \qquad \forall x \in T, 
\end{align*} where $C$ is as in \eqref{bi-lip}. In particular, it follows that $J_f, 1/J_f \in L^{\infty}(T)$ and thus $J_f \in A_p(\mu)$ for every $p \in [1,\infty]$.
\end{proposition}
 \begin{proof}
     
Assume that $x,y$ are vertices in $T$ such that $y\le x$ and $y \ne x$. Then, 
\begin{align*}
   e^{-C}\le \frac{\rho(x,y)}{\rho(f(x),f(y))}=e^{\ell(x)-\ell(f(x)\wedge f(y))} \le e^{\ell(x)-\ell(f(x))}.
\end{align*} It follows that $\ell(x)-\ell(f(x)) \ge \log e^{-C}$ for every $x \in T$. Conversely, pick $y$ such that $f(y)\le f(x)$ and $f(y) \ne f(x)$ (this is possible since $f$ is surjective). Then, 
\begin{align*}
  e^{\ell(x)-\ell(f(x))} \le \frac{\rho(x,y)}{\rho(f(x),f(y))}=e^{\ell(x \wedge y)-\ell(f(x)) } \le e^C,
\end{align*} from which follows $\ell(x)-\ell(f(x)) \le  C$.
Therefore, 
\begin{align*}
   -C \le  \ell(x)-\ell(f(x))\le  C \qquad \forall x \in T,
\end{align*} as desired. 
 \end{proof} 
 It turns out that bilipschitz bijections with respect to $\rho$ are quasi-isometries with respect to $d$.
 
 \begin{proposition}
 Suppose that $f$ is bilipschitz on $T$ with respect to $\rho$. Then $f$ is a $(1,4C)$-quasi-isometry bijection with respect to $d$, namely 
 $$
 d(f(x),f(y))-4C \leq d(x,y)\leq d(f(x),f(y)) +4C,
 $$
 for all $x,y\in T$. In particular, every isometry with respect to $\rho$ is an isometry with respect to $d$.
 \end{proposition}
 \begin{proof}
 Proposition~\ref{bilipAp} implies that
 $$
 -C\leq \ell(x)-\ell(f(x))\leq C.
 $$
 Moreover, \eqref{bi-lip} readily implies that
 $$
 -C\leq \ell(x\wedge y) - \ell(f(x)\wedge f(y))\leq C.
 $$ 
 On the one hand, for every $x,y\in T$,
 \begin{align*}
 d(x,y) &= 2\ell(x\wedge y)-\ell(x)-\ell(y) \\
 &\leq 2\ell(x\wedge y) +C-\ell(f(x))+C-\ell(f(y))+2\ell(f(x)\wedge f(y))-2\ell(f(x)\wedge f(y))\\
 &=4C +d(f(x),f(y)).
  \end{align*}
 On the other hand,
  \begin{align*}
  d(f(x),f(y)) &= 2\ell(f(x)\wedge f(y)) -\ell(f(x))-\ell(f(y))\\
  &\leq 2\ell(f(x)\wedge f(y)) +C-\ell(x)+C-\ell(y)+2\ell(x\wedge y)-2\ell(x\wedge y)\\
  &=4C +d(x,y).
 \end{align*}
 The statement about the isometries follows by setting $C=0$.
 \end{proof}

 The latter two propositions identify a class of quasi-isometric bijections on $(T,d,\mu)$ whose Jacobian is in $L^\infty$. In a subsequent paper, we plan to investigate the statement that quasisymmetric mappings on $(T,\rho,\mu)$ are a subset of $(L,C)$-quasi-isometric bijections on $(T,d,\mu)$, and that
 these quasi-isometries have Jacobian in $A_p(\mu)$ for some $p$.

 \bibliographystyle{plain}
{\small
\bibliography{references}}
\end{document}